\documentclass[11pt]{amsart}
\usepackage{amssymb,amstext,amsmath,amscd,amsthm,amsfonts,enumerate,latexsym, comment}
\usepackage[dvipsnames]{xcolor}
\usepackage{amsfonts,amsmath,amssymb,amsthm,amscd,amsxtra}
\usepackage{mathtools,mathptmx}
\usepackage{enumerate,epsfig,verbatim}
\usepackage{centernot}
\usepackage{todonotes}
\usepackage[T1]{fontenc}
\usepackage[usenames,dvipsnames]{pstricks}
\usepackage[mathscr]{eucal}
\usepackage[notcite,notref]{}
\usepackage{tikz}
\usetikzlibrary{graphs,graphs.standard,matrix, calc, arrows,quotes}
\usepackage{relsize}
\usepackage[all,2cell,ps]{xy}
\usepackage[pagebackref]{hyperref}
\usepackage{nicefrac}
\usepackage{array}
\usepackage{hyperref} 
\usepackage[all]{xypic}
\usepackage[notcite,notref]{}
\usepackage{url}
\usepackage{datetime}

\theoremstyle{plain}
\newtheorem{thm}{Theorem}[section]

\newtheorem*{thm*}{Theorem}
\newtheorem*{cor*}{Corollary}
\newtheorem{prop}[thm]{Proposition}

\newtheorem{lem}[thm]{Lemma}
\newtheorem{cor}[thm]{Corollary}

\newtheorem*{claim*}{Claim}

\theoremstyle{definition}
\newtheorem{defn}[thm]{Definition}

\newtheorem{ex}[thm]{Example}

\newtheorem{rem}[thm]{Remark}

\newtheorem{que}[thm]{Question}

\theoremstyle{remark}

\newtheorem{chunk}[thm]{\hspace*{-1.065ex}\bf}
\numberwithin{equation}{thm}

\newcommand{\rmE}{\mathrm{E}}

\newcommand{\rmQ}{\mathrm{Q}}

\newcommand{\fm}{\mathfrak{m}}
\newcommand{\Ko}{\operatorname{K}}

\def\Ann{\operatorname{Ann}}

\def\depth{\operatorname{depth}}

\def\Ext{\operatorname{Ext}}
\def\End{\operatorname{End}}
\def\len{\operatorname{length}}

\def\Hom{\operatorname{Hom}}

\def\e{\operatorname{e}}
\def\fm{\mathfrak{m}}

\def\fq{\mathfrak{q}}

\def\im{\operatorname{im}}

\def\tr{\operatorname{tr}}

\title[Full-trace modules]{Full-trace modules}

\author[Ela Celikbas]{Ela Celikbas}
\address{Ela Celikbas\\
School of Mathematical and Data Sciences\\
West Virginia University \\
Morgantown, WV 26506-6310, USA}
\email{ela.celikbas@math.wvu.edu}

\author[Olgur Celikbas]{Olgur Celikbas}
\address{Olgur Celikbas\\
School of Mathematical and Data Sciences\\
West Virginia University \\
Morgantown, WV 26506-6310, USA and Max Planck Institute for Mathematics, Vivatsgasse 7, 53111 Bonn, Germany}
\email{olgur.celikbas@math.wvu.edu}

\author[J\"{u}rgen Herzog]{J\"{u}rgen Herzog}
\address{J\"{u}rgen Herzog\\
Fachbereich Mathematik, Universit\"at Duisburg-Essen, Fakult\"at f\"ur Mathematik, 45117 Essen, Germany}
\email{juergen.herzog@uni-essen.de}

\author[Shinya Kumashiro]{Shinya Kumashiro}
\address{Shinya Kumashiro\\
Department of Mathematics, Osaka Institute of Technology, 5-16-1 Omiya, asahi-ku, Osaka, 535-8585, Japan}
\email{shinya.kumashiro@oit.ac.jp}
\email{shinyakumashiro@gmail.com}


\thanks{Professor J\"{u}rgen  Herzog has passed away in April 2024.}

\begin{document}

\begin{abstract} Motivated by the definition of nearly Gorenstein rings, we introduce the notion of \emph{full-trace modules} over commutative Noetherian local rings—namely, finitely generated modules whose trace equals the maximal ideal. We investigate the existence of such modules and prove that, over rings that are neither regular nor principal ideal rings, every positive syzygy module of the residue field is full-trace. Moreover, over Cohen–Macaulay rings, we study \emph{full-trace Ulrich modules}—that is, maximally generated maximal Cohen–Macaulay modules that are full-trace. We establish the following characterization: a non-regular Cohen–Macaulay local ring has minimal multiplicity if and only if it admits a full-trace Ulrich module. Finally, for numerical semigroup rings with minimal multiplicity, we show that each full-trace Ulrich module decomposes as the direct sum of the maximal ideal and a module that is either zero or Ulrich.
\end{abstract}

\maketitle

\section{Introduction}\label{section1}

Throughout, $R$ denotes a commutative Noetherian ring, and all $R$-modules are assumed to be finitely generated. If $R$ is local, $\fm$ and $k$ denote the maximal ideal and the residue field of $R$, respectively.

The main object of study in this paper is the trace of $R$-modules. Setting $M^{\ast} = \Hom_R(M, R)$, we define the \textit{trace} $\tr_R(M)$ of an $R$-module $M$ as the ideal of $R$ given by $\sum_{f \in M^{\ast}} \operatorname{im}(f)$.

The concept of trace ideals in commutative algebra dates back to the late 1960s. One of the earliest definitions appears in Vasconcelos's 1969 paper \cite{Vas69}, where the trace of a (projective) $R$-module $M$ is defined as the image of the evaluation map $M \otimes_R M^{\ast} \to R$, given by $x \otimes f \mapsto f(x)$. Vasconcelos revisited the notion in his 1973 paper \cite{Vas73}, interpreting it as we have presented above.

Trace ideals have proven to be a powerful tool in commutative algebra. For example, in the study of Gorenstein rings, it has been shown that if $R$ is local, then $R$ is Artinian Gorenstein if and only if every ideal of $R$ is a trace ideal; that is, every ideal equals its own trace; see \cite[1.2 and 2.2]{Lindo2}. Consequently, there has been considerable interest in understanding rings in which every ideal is a trace ideal, a question investigated in detail in \cite{TTTI}. Trace ideals have also been used in the classification of rings. For example, it was established in \cite[7.4]{LL} that a one-dimensional local domain with infinite residue field is an Arf ring \cite{Lipman} if and only if every trace ideal is stable. Moreover, trace ideals have been applied to the study of several well-known conjectures. For example, they have played a role in work \cite{BergerSarasij} concerning Berger's conjecture. Similarly, Lindo’s result \cite[6.8]{Lindo1}, in view of \cite[8.6]{CeRo}, implies that the vanishing condition of the celebrated Auslander--Reiten Conjecture \cite{AuBr} holds for trace ideals over Gorenstein domains of arbitrary dimension.

The trace of the canonical module carries important information; see, for example, \cite{DingIndex}. Herzog--Hibi--Stamate studied the trace of the canonical module and proved that it describes the non-Gorenstein locus of \( R \). Therefore, if \( R \) is a Cohen--Macaulay local ring, then \( R \) is not Gorenstein if and only if \( \tr_R(\omega_R) \subseteq \fm \); see \cite[2.1]{TraceCanonical}. Among the various results of Herzog--Hibi--Stamate is the introduction of the nearly Gorenstein property: a Cohen--Macaulay local ring \( R \) with canonical module \( \omega_R \) is called {\it nearly Gorenstein} if \( \fm \subseteq \tr_R(\omega_R) \). If \( R \) is local and \( M \) is an \( R \)-module, then \( \tr_R(M) = R \) if and only if \( M \) has a free summand \cite[2.8(iii)]{Lindo1}. Thus, if \( R \) is nearly Gorenstein but not Gorenstein, then $\fm \subseteq \tr_R(\omega_R)\neq R$, that is, $\tr_R(\omega_R)=\fm$. This motivates the study of modules that have no free summand but whose trace is as large as possible. We say that an $R$-module $M$ is \emph{full-trace} if $\tr_R(M)=\fm$.

If \( R \) is not regular, then \( \fm \) is a full-trace \( R \)-module, and hence \( R \) admits infinitely many full-trace \( R \)-modules; see \ref{gozlemle-1}(i) and \ref{gozlemle-2}(ii). Since \( \fm \) is the first syzygy of \( k \), this raises the question of whether other syzygies of \( k \) are also full-trace. To explore this, we consider the following:

\begin{que} \label{mainsoru} 
If \( R \) is a local ring and \( n \geq 1 \) is a given integer, is \( \Omega^n_R(k) \) a full-trace \( R \)-module? 
\end{que}

The main purpose of this paper is to provide a complete answer to Question \ref{mainsoru}. We show that there is no full-trace \( R \)-module when \( R \) is a discrete valuation ring. On the other hand, if \( R \) is not a discrete valuation ring, we determine exactly which syzygy modules of the residue field are full-trace and which are not. Our main theorem is the following:

\begin{thm} \label{intro-thm-1} Let $R$ be a local ring. 
\begin{enumerate}[\rm(1)] 
\item Assume $R$ is regular. Set $d=\dim(R)$.
\begin{enumerate}[\rm(i)] 
\item If $d=1$, then there is no full-trace $R$-module.
\item If $d\geq 2$, then $\Omega_R^{i}(k)$ is a full-trace $R$-module for all $i=1, \ldots, d-1$, and $\Omega_R^{i}(k)$ is not a full-trace $R$-module for all $i\geq d$.
\end{enumerate}
\item Assume $R$ is a non-regular principal ideal ring.
\begin{enumerate}[\rm(i)] 
\item If $\fm^2=0$, then $\Omega_R^{i}(k)$ is a full-trace $R$-module for all $i\geq 0$. 
\item If $\fm^2\neq 0$, then $\Omega_R^{2i+1}(k)$ is a full-trace $R$-module for all $i\geq 0$, and $\Omega_R^{2i+2}(k)$ is not a full-trace $R$-module for all $i\geq 0$.
\end{enumerate}
\item If $R$ is a non-regular, non-principal ideal ring, then $\Omega_R^{i}(k)$ is a full-trace $R$-module for all $i\geq 0$. 
\end{enumerate}
\end{thm}

We give a proof of Theorem \ref{intro-thm-1} in Section \ref{FT}; see also Proposition \ref{mainthm-1}. A key component in the proof of part (3) of Theorem \ref{intro-thm-1} is the construction of the minimal free resolution of the residue field \( k \), established by Gulliksen \cite{Gu} and Tate \cite{Tate}.

In the second part of the paper, namely in Section~\ref{FU}, we consider, over Cohen-Macaulay local rings, \emph{full-trace Ulrich modules}---that is, modules that are both full-trace and maximally generated maximal Cohen-Macaulay. Motivated by Theorem~\ref{intro-thm-1} and a result of Brennan--Herzog--Ulrich~\cite{BHU}, we are led to investigate a possible connection between Cohen-Macaulay local rings of minimal multiplicity and the full-trace Ulrich property. In this direction, we prove the following result; see Proposition~\ref{prop-min-mul}, Corollary~\ref{cor-equivalent}, and the beginning of Section~\ref{FU} for further details.

\begin{prop} \label{intro-thm-2} Let $R$ be a $d$-dimensional non-regular Cohen-Macaulay local ring. Then the following conditions are equivalent.
\begin{enumerate}[\rm(i)] 
\item $R$ has minimal multiplicity. 
\item There exists a full-trace Ulrich $R$-module.
\end{enumerate}
\end{prop}

We specialize to one-dimensional Cohen--Macaulay local rings in the last part of Section~\ref{FU} and prove the following result:

\begin{prop} \label{cor-end-3-intro} Let $R$ be a numerical semigroup ring which has minimal multiplicity. Given an $R$-module $M$, the following conditions are equivalent:
\begin{enumerate}[\rm(i)]
\item $M$ is full-trace Ulrich.
\item $M \cong \fm \oplus N$ for some $R$-module $N$ which is either zero or an Ulrich $R$-module.
\end{enumerate}
\end{prop}

Proposition~\ref{cor-end-3-intro} follows from Proposition~\ref{2ndthm}, a more general result that shows the same decomposition result applies to all one-dimensional Cohen-Macaulay local rings, provided that their endomorphism algebra \( \operatorname{Hom}_R(\mathfrak{m}, \mathfrak{m}) \) is a local ring. We present an example of a two-dimensional Cohen--Macaulay local ring \( R \) and a full-trace Ulrich \( R \)-module that does not decompose as in Proposition~\ref{cor-end-3-intro}; see Example~\ref{sonornek}. This demonstrates that the one-dimensional hypothesis is necessary.

\section{Proof of Theorem \ref{intro-thm-1}}\label{FT}

In this section, we prove Theorem \ref{intro-thm-1}. Along the way, we recall some necessary background and develop a number of preliminary results. 

\begin{chunk} \label{gozlemle-1} Let $R$ be a ring.
\begin{enumerate}[\rm(i)]
\item If $I$ is an ideal of $R$, then the natural inclusion $I\hookrightarrow R$ implies that $I \subseteq \tr_R(I)$; see \cite[2.8(iv)]{Lindo1}.
\item If $M$ and $N$ are $R$-modules, then $\tr_R(M\oplus N)=\tr_R(M) + \tr_R(N)$; see \cite[2.8(ii)]{Lindo1}. Thus, if $R$ is local, $M$ is a full-trace $R$-module, and $N$ is an $R$-module with no free direct summand, then $M\oplus N$ is a full-trace $R$-module. So, if $R$ is local and admits a full-trace $R$-module, then $R$ admits infinitely many full-trace $R$-modules.
\end{enumerate}
\end{chunk}

\begin{chunk} \label{gozlemle-2} Let $R$ be a local ring.  
\begin{enumerate}[\rm(i)]
\item If $\Omega_R^i(k)$ has a free summand for some $i\geq 0$, then $R$ is regular; see \cite[1.3]{Dutta}.
\item Let $M$ be an $R$-module. As mentioned in the introduction, $\tr_R(M)=R$ if and only if $M$ has a free summand. So, if $M$ has no free summand and $\fm \subseteq \tr_R(M)$, then $\tr_R(M)=\fm$, that is, $M$ is a full-trace $R$-module; see \cite[2.8(iii)]{Lindo1}.
\item Assume $R$ is not regular. If $\fm \subseteq \tr_R\big(\Omega_R^i(k)\big)$ for some $i\geq 0$, then $\Omega_R^i(k)$ is a full-trace $R$-module due to parts (ii) and (iii). Therefore, $\fm \oplus \Omega_R^i(k)$ is a full-trace $R$-module for all $i\geq 0$; see \ref{gozlemle-1}(ii).
\end{enumerate}
\end{chunk}

The second part of Theorem~\ref{intro-thm-1}, stated in the introduction, is subsumed by the following more general result.

\begin{prop} \label{mainthm-1} Let $R$ be a non-regular principal ideal ring. There exists $x\in R$ and a positive integer $n$ such that $\fm=Rx$ and $\fm^n=0$. Then the minimal free resolution of $k$ has the following form
\[
\cdots \xrightarrow{x^{n-1}} R \xrightarrow{x} R \xrightarrow{x^{n-1}} R \xrightarrow{x} R \to 0.
\]
Moreover, $\Omega_R^{2i+1} (k) \cong \fm=\tr_R\big(\Omega_R^{2i+1}(k)\big)$ and $\Omega_R^{2i+2} (k) \cong \fm^{n-1}=\tr_R\big(\Omega_R^{2i+2}(k)\big)$, for all $i\geq 0$. 
\end{prop}

\begin{proof} As $R$ is a principal ideal ring, there exists $x\in R$ such that $\fm=Rx$. Note that $R$ is Artinian since $R$ is not regular and its embedding dimension is at most one. Thus, $\fm^n=0$ for some $n\geq 2$. We may assume $n$ is the smallest such positive integer. As every ideal of $R$ is a power of $\fm$, it follows by the choice of $n$ that $(0:_Rx)=(x^{n-1})=\fm^{n-1}$. This implies that the minimal free resolution of $k$ has the following form:
\[
\cdots \xrightarrow{x^{n-1}} R \xrightarrow{x} R \xrightarrow{x^{n-1}} R \xrightarrow{x} R \to k \to 0.
\]
Therefore, it follows that 
\begin{align*}
\Omega_R^{i} (k) \cong 
\begin{cases}
\fm  & \text{if $i\geq 1$ and $i$ is odd}\\
\fm^{n-1} & \text{if $i\geq 2$ and $i$ is even.}
\end{cases}
\end{align*}
So, if $i\geq 0$, then $\Omega_R^{2i+1} (k)=\fm$ is a full-trace $R$-module; see \ref{gozlemle-2}(iii). Note that, if $f\in \Hom_R\big((0:_Rx), R\big)$, then $\im(f) \subseteq (0:_Rx)$. This implies that $\tr_R\big((0:_Rx)\big) \subseteq (0:_Rx)$. Therefore, $\tr_R\big((0:_Rx)\big)=(0:_Rx)$, that is, $\tr_R \big( \Omega_R^{2i+2}(k) \big)=\tr_R(\fm^{n-1})=\fm^{n-1}$; see \ref{gozlemle-1}(i). Consequently, if $\fm^2=0$, then $\Omega_R^{i}(k)$ is a full-trace $R$-module for all $i\geq 0$. Conversely, if $\fm^2\neq 0$, then $\Omega_R^{2i+1}(k)$ is a full-trace $R$-module for all $i\geq 0$, while $\Omega_R^{2i+2}(k)$ is not a full-trace $R$-module for all $i\geq 0$.
\end{proof}

Let $\varphi:F\to G$ be an $R$-module homomorphism, where $F$ and $G$ are free $R$-modules. Then $\varphi$ is given by a matrix $U$ with respect to bases of $F$ and $G$; we set $I(\varphi)$ to be the ideal of $R$ generated by the entries of $A$. Note that $I(\varphi)$ only depends on $\varphi$, and does not depend on the choice of bases of $F$ and $G$. The following related lemma plays an important role in the proof of parts (1) and (3) of Theorem \ref{intro-thm-1}.

\begin{lem}\label{l24} Let $R$ be a ring, and let $\varphi:F\to G$ be an $R$-module homomorphism, where $F$ and $G$ are free $R$-modules. Then $I(\varphi) \subseteq \tr_R\big(\im (\varphi)\big)$. 
\end{lem}

\begin{proof} Set $F=R^{\oplus n}$ and $G=R^{\oplus m}$ for some positive integers $m$ and $n$, and let $\{f_1, \dots, f_n\}$ and $\{g_1, \dots, g_m\}$ be $R$-bases of $F$ and $G$, respectively. We consider the $m\times n$ matrix $A=(a_{ij})$ representing $\varphi$ with respect to the chosen bases. Note that the $R$-submodule $\im (\varphi)$ of $G$ is generated by the column vectors 
\begin{center}
\[
\scriptsize
\left(\begin{matrix} a_{11} \\ \vdots \\ a_{m1} \end{matrix}\right),
\left(\begin{matrix} a_{12} \\ \vdots \\ a_{m2} \end{matrix}\right), \dots,
\left(\begin{matrix} a_{1n} \\ \vdots \\ a_{mn} \end{matrix}\right).
\]
\end{center}

For each \( i = 1, \dots, m \), we define an \( R \)-module homomorphism \( p_i: G \to R \), where \( p_i(g_i) = 1 \) and \( p_i(g_j) = 0 \) for \( i \neq j \). Let \( \alpha: \im(\varphi) \hookrightarrow G \) be the natural injection. For each \( i = 1, \dots, m \), we have an \( R \)-module homomorphism \( p_i\alpha: \im(\varphi) \to R \) such that \( \im(p_i\alpha) \) is the ideal of \( R \) generated by \( a_{i1}, \dots, a_{in} \). As \( \im(p_i\alpha) \subseteq \tr_R(\im(\varphi)) \) for each \( i = 1, \dots, m \), we conclude that
\[
I(\varphi) = (a_{11}, \dots, a_{1n}, a_{21}, \dots, a_{2n}, \dots, a_{m1}, \dots, a_{mn}) \subseteq \tr_R(\im(\varphi)),
\]
as claimed.
\end{proof}

Parts (ii) and (iii) of the next lemma are essentially from  \cite[2.8(iii)]{Lindo1}.

\begin{lem} \label{lemmatrace} Let $R$ be a ring and $M$ be an $R$-module. Then the following hold:
\begin{enumerate}[\rm(i)]
\item There is a surjective $R$-module homomorphism $M^{\oplus n} \twoheadrightarrow \tr_R(M)$ for some $n\geq 1$.
\item If $\tr_R(M) \cong R$, then $\tr_R(M)=R$ and $M^{\oplus n} \cong R \oplus N$ for some $R$-module $N$ and for some $n\geq 1$.
\item If $R$ is local and $\tr_R(M) \cong R$, then $M \cong R\oplus T$ for some $R$-module $T$.
\end{enumerate}
\end{lem}

\begin{proof} To prove part (i), assume $\tr_R(M)=(x_1, \ldots, x_s)$ for some $x_i\in R$. Then, by the definition of trace, for all $i=1, \ldots, s$, we have that
$
x_i=r_1^i f_1^i(y_1^i)+\ldots+r_{n_i}^i f_{n_i}^i(y_{n_i}^i)
$
for some integer $n_i\geq 1$, $r_j^i \in R$, $y^i_j \in M$, and $f^i_j \in M^{\ast}$. Setting $\Psi_i=f_1^i+\ldots+f_{n_i}^i$ and $\beta_i=(r_1^i y_1^i, \ldots, r_{n_i}^i y_{n_i}^i) \in M^{\oplus n_i}$, we see that $\Psi_i(\beta_i)=x_i$ for each $i$. Consequently, $\Psi_1+ \cdots + \Psi_{n_s}$ is a surjection from $M^{\oplus n_1}\oplus \cdots \oplus M^{\oplus n_s}$ to $\tr_R(M)$.

For part (ii), consider a surjection $M^{\oplus s}  \twoheadrightarrow \tr_R(M)$ for some $s\geq 1$. Then, as $\tr_R(M) \cong R$, we get a surjection $\psi: M^{\oplus s}  \twoheadrightarrow R$. Hence, there is an element $(y_1, \ldots, y_s)\in M^{\oplus s}$ such that $\psi(y_1, \ldots, y_s)=1$. Let $i_1, \ldots, i_s$ be the natural inclusions $M \hookrightarrow M^{\oplus s}$. Then we have
$$1=\psi(y_1, \ldots, y_s)=(\psi i_1)(y_1)+\cdots + (\psi i_s)(y_s)\subseteq \im(\psi i_1)+\cdots+\im(\psi i_s)\subseteq \tr_R(M),$$
which shows that $\tr_R(M)=R$. In that case, the surjection stated in part (i) splits and hence yields an isomorphism $M^{\oplus n} \cong R \oplus N$ for some $R$-module $N$.

Assume $R$ is local and $\tr_R(M)=R$. Then there exist $\psi_1, \ldots, \psi_r \in M^{\ast}$ and $x_1, \ldots, x_r\in M$ such that $\psi_1(x_1)+\cdots+\psi_r(x_r)=1$. Since $R$ is local, at least one of $\psi_i(x_i)$ must be a unit in $R$, which implies that the corresponding map $\psi_i: M \to R$ is surjective and hence splits.
\end{proof}

We now proceed to prove the first part of Theorem~\ref{intro-thm-1}, using the following fact:

\begin{chunk} [{\cite[3.3]{TakSyz}}] \label{Takahashi} Let $R$ be a local ring such that $d=\depth(R)$. If $d\geq 1$, then $\Omega_R^{i}(k)$ is indecomposable for all $i=0, \ldots, d-1$.
\end{chunk}

\begin{proof}[Proof of part (1) of Theorem~\ref{intro-thm-1}]
We first assume $R$ is a one-dimensional regular local ring and suppose there is a full-trace $R$-module $M$. Then $\tr_R(M)=\fm$. This implies that $\tr_R(M)=R$ and gives a contradiction due to Lemma \ref{lemmatrace} because $\fm \cong R$. Consequently, there is no full-trace $R$-module.
This proves the theorem for the case where $R$ is a one-dimensional regular local ring.

Next, assume $R$ is a $d$-dimensional regular local ring with $d\geq 2$. Then $\Omega_R^{i}(k)=0$ for all $i\geq d+1$ so that $\tr_R\big(\Omega_R^{i}(k)\big)=0$. Similarly, $\tr_R\big(\Omega_R^{d}(k)\big)=R$ since $\Omega_R^{d}(k)$ is free. Therefore, $\Omega_R^{i}(k)$ is not a full-trace $R$-module for all $i\geq d$. Next, we consider $\Omega_R^{i}(k)$ for all $i=1, \ldots, d-1$.

Let $\underline{x}$ be a regular system of parameters. Then the Koszul complex $\Ko(\underline{x};R)$ yields a minimal free resolution of $R/\underline{x}R=k$. Let $\partial_i$ be the $i$th differential of the complex $\Ko(\underline{x};R)$. Then Lemma \ref{l24} implies that $I(\partial_i) \subseteq \tr_R\big(\im (\partial_i)\big)$, that is, $\fm=(\underline{x})=I(\partial_i) \subseteq \tr_R\big(\im (\partial_i)\big)=\tr_R\big(\Omega_R^{i}(k)\big)$
for all $i=1, \ldots, d-1$. We know that $\Omega_R^{i}(k)$ has no free direct summand for all $i=1, \ldots, d-1$; see \ref{Takahashi}. Therefore, $\Omega_R^{i}(k)$ is a full-trace $R$-module for all $i=1, \ldots, d-1$; see \ref{gozlemle-2}(iii). This completes the proof of the first part of Theorem \ref{intro-thm-1}.
 \end{proof}
 

The proof of the third part of Theorem~\ref{intro-thm-1} is similar to that of the first part but relies on a construction of Gulliksen \cite{Gu} and Tate \cite{Tate} concerning the minimal free resolution of the residue field. Before proving the theorem, we briefly recall relevant details of their construction; see \cite[Sect. 1 and 2]{Tate}. 

\begin{chunk} (\cite{Tate}) An $R$-algebra $X$ is an associative $R$-algebra  equipped with a grading $X=\bigoplus_{n\in \mathbb{Z}} X_n$ and an $R$-module homomorphism $\partial: X\to X$ such that $X_i$ is a finitely generated $R$-module for all $i\ge 0$, $X_0 = R$, $X_i=0$ for all $i<0$ and $X$ is a skew derivation of degree $-1$ which is strictly skew-commutative (this means $\partial(X_n) \subseteq X_{n-1}$ for all $n\in \mathbb{Z}$ and $\partial^2 = 0$. Moreover, given $x \in X_{v}$ and $y \in X_{w}$, it follows that $xy=(-1)^{vw}yx$, $\partial(xy) = \partial(x) y +(-1)^v x \partial(y)$, and $x^2=0$ if $v$ is odd). 

Note that the ring $R$ can be viewed as an $R$-algebra in the trivial sense. Another typical example of such an algebra is the Koszul complex. 

An $R$-algebra $(X, \partial)$ can be regarded as a complex of $R$-modules:
\[
\cdots \to X_n \xrightarrow{\partial_n} X_{n-1} \xrightarrow{\partial_{n-1}} \cdots \to X_1  \xrightarrow{\partial_{1}} X_0 \to 0. 
\]

Given an $R$-algebra  $(X, \partial)$ and a cycle $s\in Z_n=\ker(\partial_n)$, we denote by $X\langle S; \partial(S)=s\rangle$ the extended $R$-algebra obtained by the adjunction of a variable $S$ to kill the cycle $s$. We skip the details of the construction of the extended algebra, but present the form of the extended complex constructed depending on $\deg(s)$, where $\deg(S)=\deg(s)+1$.
\vspace{0.05in}

If $\deg(s)$ is even, the complex $X\langle S; \partial(S)=s\rangle$ can be regarded as:
\[
\cdots \to 
\begin{matrix}
X_{n} \\ \oplus \\ X_{n-p}
\end{matrix}
\xrightarrow{\left(
\begin{smallmatrix}
\partial_n & *\\
0 & \partial_{n-p}
\end{smallmatrix}\right)} 
\begin{matrix}
X_{n-1} \\ \oplus \\ X_{n-1-p} 
\end{matrix}
\to \cdots \to 
\begin{matrix}
X_p \\
\oplus \\
X_0 
\end{matrix}
\to X_{p-1} \to \cdots \to X_0 \to 0.
\]

\vspace{0.05in}

If $\deg(s)$ is odd, the complex $X\langle S; \partial(S)=s\rangle$ obtained from $X$ can be regarded as:
\[
\cdots \to 
\begin{matrix}
X_{n} \\ \oplus \\ X_{n-p} \\ \oplus \\ X_{n-2p} \\ \oplus \\ \vdots 
\end{matrix}
\xrightarrow{\left(
\begin{smallmatrix}
\partial_n & * & 0 & \cdots \\
0 & \partial_{n-p} & * & \ddots \\
0 & 0 & \partial_{n-2p} & \ddots \\
\vdots & \vdots & \ddots & \ddots  
\end{smallmatrix}\right)} 
\begin{matrix}
X_{n-1} \\ \oplus \\ X_{n-1-p} \\ \oplus \\ X_{n-1-2p} \\ \oplus \\ \vdots 
\end{matrix}
\to \cdots \to 
\begin{matrix}
X_p \\
\oplus \\
X_0  
\end{matrix}
\to X_{p-1} \to \cdots \to X_0 \to 0.
\]
(Note that the direct sum $X_n\oplus X_{n-p} \oplus X_{n-2p} \oplus \cdots$ is a finite sum since $X_{n-\ell p} =0$ when $\ell p>n$.)
\end{chunk}

Next is the proof of the third part of Theorem \ref{intro-thm-1}.

\begin{proof}[Proof of part (3) of Theorem \ref{intro-thm-1}] Assume $R$ is a local ring such that $R$ is not regular and not a principal ideal ring. Let $\underline{x}=\{x_1, \ldots, x_n\}$ be a minimal generating set of the maximal ideal $\fm$ and let $(K, \partial)$ be the Koszul complex on $\underline{x}$. Note that $n\geq 2$ since $R$ is not a principal ideal ring. Note also that $I(\partial_1)=\fm=I(\partial_2)$.

If $H_1(K)=0$, then $\underline{x}$ is a regular sequence so that $R$ is regular; see \cite[1.6.19, 2.2.5]{BH}. Therefore, $H_1(K)\neq 0$. Hence, we proceed by obtaining the extended algebras by killing the elements in $H_1(K)$ represented by the cycles. Given such a cycle $s_{n+1}$, as its degree is two, the algebra $K\langle S; \partial(S_{n+1})=s_{n+1}\rangle$ extended from $K$ can be regarded as follows:

\[
\cdots \to 
\begin{matrix}
K_{n} \\ \oplus \\ K_{n-2} \\ \oplus \\ K_{n-4} \\ \oplus \\ \vdots 
\end{matrix}
\xrightarrow{\left(
\begin{smallmatrix}
\partial_n' & * & 0 & \cdots \\
0 & \partial_{n-2}' & * & \ddots \\
0 & 0 & \partial_{n-4}' & \ddots \\
\vdots & \vdots & \ddots & \ddots  
\end{smallmatrix}\right)} 
\begin{matrix}
K_{n-1} \\ \oplus \\ K_{n-3} \\ \oplus \\ K_{n-5} \\ \oplus \\ \vdots 
\end{matrix}
\to \cdots \to 
\begin{matrix}
R^{\binom{n}{2}}\\
\oplus \\
R 
\end{matrix}
\to R^{\oplus n} \to R \to 0. 
\]

In this way we obtain an extended algebra $Y=K\langle S; \partial(S_{n+1})=s_{n+1}, \ldots, \partial(S_{n+r})=s_{n+r}\rangle$ for some $r\geq 1$ with $H_1(Y)=0$.
Subsequently, we continue killing the cycles in $H_2(Y)$. Eventaully, by \cite[Thm. 1]{Tate} and \cite[Thm.]{Gu}, we obtain a minimal free resolution $K\langle \ldots, S_i, \ldots; \partial(S_{i})=s_{i}\rangle$ of $k$ with differential $\varphi$ such that  $\varphi_1=\partial_1$, $I(\partial_2)\subseteq I(\varphi_2)$, and $I(\varphi_i)$ contains $I(\partial_1)$ or $I(\partial_2)$ for all $i\geq 3$ (this process may require infinitely many steps.) So, in view of Lemma \ref{l24}, the following inclusions hold for all $i\geq 1$:
$$\fm \subseteq I(\varphi_i) \subseteq \tr_R\big(\im (\varphi_i)\big) \subseteq \tr_R\big( \Omega_R^{i}(k)\big)$$
Thus, we conclude from \ref{gozlemle-2}(iii) that $\tr_R\big( \Omega_R^{i}(k)\big)=\fm$ for all $i\geq 1$.
\end{proof}

\section{Modules which are both full-trace and Ulrich}\label{FU}

Let \( R \) be a Cohen-Macaulay local ring and let \( M \) be a maximal Cohen-Macaulay \( R \)-module, that is, \( M \neq 0 \) and \( \depth_R(M) = \dim_R(M) = \dim(R) \). Then \( \mu_R(M) = \len_R(M / \fm M) \leq \e_R(M) \), where \( \e_R(M) \) denotes the (Hilbert-Samuel) multiplicity of \( M \) (with respect to \( \fm \)); see, for example, \cite[page 26]{UlrichMZ}. Brennan, Herzog, and Ulrich \cite{BHU} and Ulrich \cite{UlrichMZ} studied maximal Cohen-Macaulay modules that are maximally generated, namely maximal Cohen-Macaulay \( R \)-modules \( M \) such that \( \mu_R(M) = \e_R(M) \). These modules are now known as \emph{Ulrich modules} in the literature.

The theory of Ulrich modules is an active area of research, and these modules have been shown to have important connections with several conjectures and open problems in commutative algebra and algebraic geometry. Although Ulrich modules do not always exist, their existence has significant implications for the homological properties of the base ring and its modules; see \cite{Farrah}. A result of particular interest to us in this paper is due to Brennan--Herzog--Ulrich~\cite{BHU}, which states that if \( R \) has minimal multiplicity, that is, if \( e(R) = \mu_R(M) - \dim(R) + 1 \), then \( \Omega^i_R(k) \) is an Ulrich \( R \)-module for all \( i \geq \dim(R) \).

Recall that Theorem~\ref{intro-thm-1} shows that \( \Omega^i_R(k) \) is a full-trace \( R \)-module in many cases. Motivated by the aforementioned result of Brennan--Herzog--Ulrich, we investigate a possible connection between minimal multiplicity and the property of being both full-trace and Ulrich. Our main results in this direction are as follows: If \( R \) has minimal multiplicity, then \( \Omega^i_R(k) \) is both full-trace and Ulrich for all \( i \geq \dim(R) \). Conversely, if \( R \) is not regular, the existence of a full-trace Ulrich \( R \)-module implies that \( R \) has minimal multiplicity. These results are established in Proposition~\ref{prop-min-mul} and Corollary~\ref{cor-equivalent}.

For the reader’s convenience, we recall some basic facts in Sections~\ref{Ulrichmul-pre}, \ref{Ulrichsyz}, \ref{Ulrichsyz2}, and~\ref{Ulrichmul}.

\begin{chunk}\label{Ulrichmul-pre} Let $R$ be a local ring, $\fq$ be an $\fm$-primary ideal of $R$, and let $M$ be a an $R$-module. Then,
\begin{enumerate}[\rm(i)] 
\item If $\fq$ is a parameter ideal, then $\e_R(\fq, M)\leq \len_R(M/\fq M)$; see \cite[14.9]{Mat}.
\item If $M$ is Cohen-Macaulay, then $\e_R(\fq, M)\geq \len_R(M/\fq M)$; see \cite[proof of 11.1.10]{HunekeSwanson}. 
\item If $\fq$ is a reduction of $\fm$, then $\e_R(\fq, M)=\e_R(M)$; see \cite[4.6.5]{BH}
\end{enumerate}
\end{chunk}
 
\begin{chunk} \label{Ulrichsyz}  \label{redvspar} Let $R$ be a Cohen-Macaulay local ring.
\begin{enumerate}[\rm(i)]
\item If $|k|=\infty$, there is a minimal reduction of $\fm$, which is necessarily a parameter ideal of $R$; see \cite[8.3.5, 8.3.7, and 8.3.9]{HunekeSwanson}. 
\item If $R$ is an Ulrich $R$-module, then $R$ is a regular ring. 
\item If $|k|=\infty$, then $R$ has minimal multiplicity if and only if $\fm^2=\fq \fm$ for some parameter ideal $\fq$ of $\fm$.
\item It follows that $\fm$ is an Ulrich $R$-module if and only if $\fm$ is a maximal Cohen-Macaulay $R$-module such that $\e(R)=\mu_R(\fm)$.
\item If $\dim(R)=1$, then $\fm$ is an Ulrich $R$-module if and only if $R$ has minimal multiplicity; see part (iv).
\end{enumerate}
\end{chunk}

\begin{chunk} \label{Ulrichsyz2} Let $R$ be a Cohen-Macaulay local ring which has minimal multiplicity. 
\begin{enumerate}[\rm(i)]
\item Assume $R$ is Artinian. Then $\fm^2=0$; see \ref{Ulrichsyz}(ii). Thus, for all $n\geq 1$, it follows that $\Omega_R^{n}(k)$ is a finite-dimensional $k$-vector space, and hence an Ulrich $R$-module by \cite[1.2]{BHU}.
\item The $R$-module $\Omega_R^{n}(k)$ is Ulrich for all $n\geq \dim(R)$. This follows from \cite[2.5]{BHU} if $R$ is not Artinian, and from part (i) if $R$ is Artinian.
\item If $R$ is a non-regular principal ideal ring, then $R$ is Artinian and hence part (i) shows that $\Omega_R^{n}(k)$ is an Ulrich $R$-module for all $n\geq 1$.
 \end{enumerate}
\end{chunk}

\begin{chunk}\label{Ulrichmul} Let $R$ be a Cohen-Macaulay local ring, $\fq$ be a parameter ideal of $R$, and let $M$ be a Cohen-Macaulay $R$-module. Assume $\fq$ is a reduction of $\fm$. Then we have the following implications: $M$ is Ulrich if and only if $\len_R(M/ \fm M)=\mu_R(M)=\e_R(M)=\len_R(M/\fq M)$ if and only if $\fm M=\fq M$; see \ref{Ulrichmul-pre}.
\end{chunk}

\begin{lem} \label{Ulrichlem1}Let $R$ be a Cohen-Macaulay local ring and let $M$ be an Ulrich $R$-module. If $\fq$ is a parameter ideal of $R$ which is a reduction of $\fm$, then $\fm \tr_R(M)=\fq \tr_R(M)$.
\end{lem}

\begin{proof} Note that there exists a surjection $M^{\oplus n} \twoheadrightarrow \tr_R(M)$ for some $n\geq 1$; see Lemma \ref{lemmatrace}(i). Tensoring this surjection with $R/\fq$, we obtain another surjection $(M/\fq M)^{\oplus n}  \twoheadrightarrow  \tr_R(M)/\fq \tr_R(M)$. Therefore, $\Ann_R\big(M/\fq M\big)=\Ann_R\big((M/\fq M)^{\oplus n}\big) \subseteq \Ann_R\big( \tr_R(M)/\fq \tr_R(M)\big)$. Recall that  $\fq M=\fm M$ due to \ref{Ulrichmul}. So, $\Ann_R(M/\fq M)=\Ann_R(M/\fm M)=\fm$.
Consequently, $\fm \subseteq \Ann_R\big(\tr_R(M)/\fq \tr_R(M)\big)$. This implies that $\fm \tr_R(M)=\fq \tr_R(M)$, as required.
\end{proof}

The next observation follows from the fact stated in \ref{Ulrichmul} and Lemma \ref{Ulrichlem1}.

\begin{cor} Let $R$ be a one-dimensional Cohen-Macaulay local ring. Then the trace of each Ulrich $R$-module is also an Ulrich $R$-module. 
\end{cor}

\begin{defn} Let $R$ be a Cohen-Macaulay local ring and let $M$ be an $R$-module. We say $M$ is \emph{full-trace Ulrich} if $M$ is an Ulrich $R$-module such that $\tr_R(M)=\fm$. 
\end{defn}

\begin{ex} \label{cor-end-1} Let $R$ be a one-dimensional non-regular Cohen-Macaulay local ring. Assume $R$ has minimal multiplicity. Then $\fm^2=x\fm$ for some $x\in \fm$. This implies that $\fm^{n} \cong \fm$ for all $n\geq 1$, and hence $\fm^{n}$ is full-trace for all $n\geq 1$. Set $e=\e(R)$. Then $\fm^{e-1}$ is an Ulrich $R$-module \cite[2.1]{BHU}. Note that $\fm$ has no free direct summand; see \ref{gozlemle-2}(i). Consequently,  $\fm \oplus \fm^{e-1}$  is a full-trace Ulrich $R$-module; see \ref{gozlemle-1}(ii) 
\end{ex}

\begin{prop} \label{prop-min-mul} Let $R$ be a Cohen-Macaulay local ring. If there exists a full-trace Ulrich $R$-module, then $R$ has minimal multiplicity. 
\end{prop}

\begin{proof} Note that, under the faithfully flat extension $R\to R[x]_{\fm R[x]}$, where $R[x]$ denotes the polynomial ring over $R$, the multiplicity and the full-trace Ulrich property are preserved; see \cite[2.8(viii)]{Lindo1}. Hence, we may pass to the ring $R[x]_{\fm R[x]}$ and assume $|k|=\infty$. Thus, we can choose a minimal reduction $\fq$ of $\fm$, which is necessarily a parameter ideal of $R$; see \ref{redvspar}(i).

Let $M$ be a full-trace Ulrich $R$-module. Then Lemma \ref{Ulrichlem1} shows that $\fm^2=\fm \tr_R(M)=\fq \tr_R(M)=\fq \fm$. Thus, $R$ has minimal multiplicity; see \ref{Ulrichsyz}(iii).
\end{proof}

The next corollary is  a consequence of Theorem \ref{intro-thm-1}, Proposition \ref{prop-min-mul}, and \cite[2.5]{BHU}.

\begin{cor} \label{cor-equivalent} Let $R$ be a $d$-dimensional non-regular Cohen-Macaulay local ring. Then the following conditions are equivalent.
\begin{enumerate}[\rm(i)] 
\item $R$ has minimal multiplicity. 
\item $\Omega^{i}_R(k)$ is a full-trace Ulrich $R$-module for all $i\geq d$.
\item $\Omega^{i}_R(k)$ is a full-trace Ulrich $R$-module for some $i\geq d$.
\item There exists a full-trace Ulrich $R$-module.
\end{enumerate}
\end{cor}

\begin{proof} The implication (i) $\Longrightarrow$ (ii) follows from \ref{Ulrichsyz2}, and the implication (iv) $\Longrightarrow$ (i) follows from Proposition \ref{prop-min-mul}. As the implications (ii) $\Longrightarrow$ (iii) $\Longrightarrow$ (iv) are trivial, the corollary holds.
\end{proof}

\begin{rem} In the literature there are important classes of rings which have minimal multiplicity. For example, each (local and one-dimensional) Arf ring has minimal multiplicity; see \cite[1.10 and 2.2]{Lipman}. Therefore, if $R$ is an Arf ring (for example, $R=k[\![t^e, t^{e+1}, \ldots, t^{2e-1}]\!]$ for some $e\geq 2$), then $\Omega_R^{n}(k)$ is full-trace Ulrich $R$-module for all $n\geq 1$; see Corollary \ref{cor-equivalent} and \cite[4.7]{Arf1}.
\end{rem}

We denote by $\rmQ$ the \emph{total ring of fractions} of $R$. Moreover, if $R$ is local, we denote by $\rmE$ the \emph{endomorphism algebra} $\Hom_R(\fm, \fm)$ of $\fm$. Note that, if $\depth(R)\geq 1$, then there is an isomorphism of rings $\beta: (\mathfrak{m} :_{\rmQ} \mathfrak{m}) \cong \rmE$, where $\beta(z)(x)=zx$ for all $z\in (\mathfrak{m} :_{\rmQ} \mathfrak{m})$ and $x\in \rmE$.

\begin{chunk} \label{endo} Let $R$ be a one-dimensional Cohen-Macaulay local ring which has minimal multiplicity. It follows from \cite[1.11]{Lipman} that there is a minimal reduction $(a)$ of $\fm$ (which is a parameter ideal of $R$) such that $\fm^2=a\fm$ and 
\[
\mathfrak{m} \overset{\alpha}{\cong} \frac{\mathfrak{m}}{a} 
= (\mathfrak{m} :_{\rmQ} \mathfrak{m}) 
\overset{\beta}{\cong}  \rmE, \;\; \text{ where } \frac{\mathfrak{m}}{a} = \left\{ \frac{x}{a} : x \in \mathfrak{m} \right\} \subseteq \rmQ, \; \text{ and } \alpha(y)=\frac{y}{a} \; \text{ for all } y\in \fm.
\]
\end{chunk}

\begin{rem} \label{endo2} Let $R$ be a local ring. 
\begin{enumerate}[\rm(i)]
\item Let $M$ be an $R$-module and let $f \in M^{\ast}$. If $\im(f) \nsubseteq \fm$, then $\im(f)=R$ so that the exact sequence $0 \to \ker(f) \to M  \xrightarrow{f} R \to 0$ splits, and hence $R$ is a direct summand of $M$.
\item Let $M$ be an Ulrich $R$-module. If $R$ is a direct summand of $M$, then $R$ is an Ulrich $R$-module since $M$ is Ulrich so that $\e(R)=1$, that is, $R$ is regular. Thus, if $R$ is not regular and $f\in M^{\ast}$, then $\im(f) \subseteq \fm$; see part (i). Consequently, if $R$ is not regular, then $M^{\ast}=\Hom_R(M, \fm)$.
\item Assume $R$ is not regular. Then $\fm$ does not have a free summand; see \ref{gozlemle-2}(i). So, dualizing the exact sequence $0 \to \fm \to R \to k \to 0$, and using part (ii), we obtain an injection $R \hookrightarrow \fm^{\ast}=\End_R(\fm)=\rmE$. This implies that every $\rmE$-module also has an $R$-module structure.
\end{enumerate}
\end{rem}

\begin{lem} \label{mainlem} Let $R$ be a one-dimensional non-regular Cohen-Macaulay local ring which has minimal multiplicity and let $M$ be an Ulrich $R$-module.
Then $M$ has an $\rmE$-module structure, and there is a minimal reduction $(a)$ of $\fm$ such that 
\[
M^{\ast} \cong a \cdot \Hom_{R}(M,\rmE), \;\; \Hom_{R}(M,\rmE)=\Hom_{\rmE}(M,\rmE), \;\text{ and } \; \tr_R(M)= a \cdot \tr_{\rmE}(M).
\]
\end{lem}

\begin{proof} It follows from \ref{endo} that there is a minimal reduction $(a)$ of $\fm$. Then, since $M$ is Ulrich, we have that $\fm M= a M$; see Lemma \ref{Ulrichlem1}.
As $\displaystyle{\rmE \cong \frac{\mathfrak{m}}{a}}$, it follows that $\displaystyle{\rmE M \cong \frac{\mathfrak{m}}{a} M= \frac{aM}{a}=M}$, and hence $M$ has an $\displaystyle{\frac{\fm}{a}}$, or equivalently, $\rmE$-module structure.

Note that $M^{\ast}=\Hom_R(M,\fm)$ as $R$ is not regular; see Remark \ref{endo2}(ii). So, we have an isomorphism
\[
M^{\ast}=\Hom_R(M,\fm) \overset{\cong}{\longrightarrow} a \cdot \Hom_R\left(M, \frac{\fm}{a}\right) \text{ given by } f \mapsto a \psi, \text{ where } \psi: M \to \frac{\fm}{a} \text{ and } \psi(y)=\frac{f(y)}{a}.
\]
Also, since $\displaystyle{\frac{\mathfrak{m}}{a}\overset{\beta}{\cong}  \rmE}$, there is an isomorphism $\displaystyle{a \cdot \Hom_R\left(M, \frac{\fm}{a}\right) \cong a \cdot \Hom_R\left(M, E\right)}$; see \ref{endo}. Therefore, we conclude that 
$\displaystyle{M^{\ast} \cong a \cdot \Hom_R\left(M, \frac{\fm}{a}\right)  \cong a \cdot \Hom_{R}(M,\rmE)}$.

Next we observe $\Hom_{R}(M,\rmE)=\Hom_{\rmE}(M,\rmE)$. We already have that $\Hom_{\rm E} \left(M, \rmE \right) \subseteq \Hom_R\left(M, \rmE \right)$ since $R$ is not regular; see \ref{endo2}(iii). Let $f\in \Hom_R\left(M, \rmE \right)$. To show $f$ is $\rmE$-linear, it is enough to show that $f$ is $\displaystyle{\frac{\fm}{a}}$-linear. For that, set $\displaystyle{v=\frac{b}{a} \in \frac{\fm}{a}}$ for some $b\in \fm$. Then, given $x\in M$, we have 
\[
avf(x)=bf(x)=f(bx)=f(avx)=af(vx), \; \text{ or equivalently, } \; vf(x)=f(vx), 
\]
where the second and fourth equalities hold since $f$ is $R$-linear (Note that $vx \in M$ and $a$ is a non zero-divisor on $E$ since it is a torsion-free $R$-module.) This proves that $f$ is $\rmE$-linear, and establishes the equality $\Hom_{R}(M,\rmE)=\Hom_{E}(M,\rmE)$.

Now, setting $M^{\vee}=\Hom_R(M,E)$, we see that the following equalities hold:
\[
a \cdot \tr_E(M)
= a \cdot \sum_{h \; \in \; M^{\vee}} \im(h)
= \sum_{\mathclap{\psi \; \in \; a \cdot M^{\vee}}} \im(\psi)
= \sum_{\gamma \; \in \; M^{\ast}} \im(\gamma)
= \tr_R(M).
\]
Note that, here the third equality holds since $M^{\ast} \cong a \cdot M^{\vee}$.
\end{proof}

As the endomorphism algebra of a numerical semigroup ring is local, Proposition \ref{cor-end-3-intro} advertised in the introduction is subsumed by the next result.

\begin{prop} \label{2ndthm} Let \( R \) be a one-dimensional non-regular Cohen--Macaulay local ring and let \( M \) be a full-trace Ulrich \( R \)-module. Then there is an isomorphism of $R$-modules
\[
M^{\oplus n} \cong \fm \oplus N
\]
for some \( n \geq 1 \) and some $E$-module \( N \), where $N=0$ or $N$ is an Ulrich $R$-module. Moreover, if the endomorphism algebra $E$ of $\fm$ is a local ring, one can assume $n=1$ in the direct sum decomposition so that $\fm$ is a direct summand of $M$ as an $R$-module.
\end{prop}

\begin{proof} Assume $M$ is a full-trace $R$-module. Then $R$ has minimal multiplicity due to Proposition \ref{prop-min-mul}. As $\tr_R(M)=\fm$, it follows from \ref{endo} and Lemma \ref{mainlem} that there is a minimal reduction $(a)$ of $\fm$ such that 
\[
E \cong \frac{\fm}{a}= \frac{\tr_R(M)}{a}= \frac{a \cdot \tr_{\rmE}(M)}{a}=\tr_{\rmE}(M).
\]
Now the claims follow from Lemma \ref{lemmatrace}(ii).
\end{proof}

We conclude this section with Example~\ref{sonornek}, which demonstrates that the conclusion of Proposition~\ref{2ndthm} may fail over higher-dimensional Cohen--Macaulay local rings. The following observation will be used in the example.

\begin{rem} \label{e-a} Let $R$ be a local ring such that $\depth(R)\geq 2$. Then, since $\Ext^1_R(k,R)=0$, applying $\Hom_R(-,R)$ to the short exact sequence $0 \to \fm \to R \to k \to 0$, we see that $R\cong \Hom_R(\fm, R) \cong (R:_{\rmQ}\fm)$. Also, since $\dim(R)\geq 2$, it follows that $(R:_{\rmQ} \fm)= (\fm:_{\rmQ} \fm)$. Thus, $R\cong (\fm:_{\rmQ}\fm) \cong \End_R(\fm)=E$.
\end{rem}

\begin{ex} \label{sonornek} Let $R=k[\![x^2, xy, y^2]\!]$, the second Veronese subring of $S=k[\![x,y]\!]$. Then $S=R \oplus M$, where $M=(x,y)$ is the $R$-submodule of $S$ generated by $x$ and $y$. It follows that $R\cong k[\![x,y,z]\!]/(xy-z^2)$, a two-dimensional hypersurface domain of minimal multiplicity two, where $E=\End_R(\fm)$ is a local ring; see Remark \ref{e-a}. Also, $R$ has finite Cohen-Macaulay type and all the indecomposable maximal Cohen-Macaulay $R$-modules, up to isomorphism, are $R$-direct summands of $S$, namely $R$ and $M$; see \cite[6.2]{BLW}. Note that $M$ corresponds to the ideal generated by $x^2$ and $xy$ in $R$, and $\e_R(R)=\e(M)=\mu_R(M)$. Hence, $M$ is an Ulrich $R$-module. 

Next consider the $R$-module homomorphisms $f:M=(x,y) \rightarrow R$ and $g:M=(x,y) \rightarrow R$, where $f(\alpha)=x\alpha$ and $g(\alpha)=y\alpha$ for all $\alpha \in M$. Then $\im(f)+\im(g)=(x^2, xy)+(xy, y^2)=\fm \subseteq \tr_R(M)$. As $M$ does not have a free summand (it is indecomposable), we conclude that $\tr_R(M)=\fm$; see \ref{gozlemle-2}(iii). Therefore, $M$ is a full-trace $R$-module. 

If $M\cong \fm \oplus N$ for some $R$-module $N$, then $N=0$ since $M$ is indecomposable; however this cannot hold since $\mu_R(M)=2$ and $\mu_R(\fm)=3$. Consequently, $M$ is a full-trace Ulrich $R$-module which does not have a decomposition as in Proposition \ref{2ndthm}.
\end{ex}

\section*{acknowledgements}
Part of this work was completed when Kumashiro visited the University of Duisburg-Essen and the West Virginia University in August 2023, and when Herzog visited the West Virginia University in April 2023.

O.\@ Celikbas is grateful to the Max Planck Institute for Mathematics in Bonn for its hospitality and financial support. E.\@ Celikbas is likewise thankful for the Institute's hospitality.

Kumashiro was supported by JSPS Kakenhi Grant Number JP24K16909.

\end{document}